\documentclass[reqno]{amsart}

\usepackage{amsmath,amsthm,amssymb,bm}
\usepackage{hyperref}
\usepackage{a4wide}
\usepackage{cleveref}
\usepackage{kotex}
\usepackage{graphicx,color}
\usepackage{tikz}
\numberwithin{equation}{section}
\usetikzlibrary{matrix,decorations.pathreplacing,arrows.meta}
\usepackage{xcolor}
\usepackage{aliascnt}
\newcommand{\newsharedtheorem}[2]{%
  \newaliascnt{#1}{thm}%
  \newtheorem{#1}[#1]{#2}%
  \aliascntresetthe{#1}%
}
\newtheorem{thm}{Theorem}[section]
\newsharedtheorem{lem}{Lemma}
\newsharedtheorem{prop}{Proposition}
\newsharedtheorem{cor}{Corollary}
\theoremstyle{definition}
\newsharedtheorem{exam}{Example}
\newsharedtheorem{defn}{Definition}
\newsharedtheorem{conj}{Conjecture}
\newsharedtheorem{question}{Question}
\newsharedtheorem{problem}{Problem}
\newsharedtheorem{remark}{Remark}

\crefname{thm}{Theorem}{Theorems}
\Crefname{thm}{Theorem}{Theorems}
\crefname{lem}{Lemma}{Lemmas}
\Crefname{lem}{Lemma}{Lemmas}
\crefname{prop}{Proposition}{Propositions}
\Crefname{prop}{Proposition}{Propositions}
\crefname{cor}{Corollary}{Corollaries}
\Crefname{cor}{Corollary}{Corollaries}
\crefname{defn}{Definition}{Definitions}
\Crefname{defn}{Definition}{Definitions}
\crefname{conj}{Conjecture}{Conjectures}
\Crefname{conj}{Conjecture}{Conjectures}
\newcommand\Col{\operatorname{Col}}
\newcommand\coinv{\operatorname{coinv}}
\newcommand\inv{\operatorname{inv}}
\newcommand\SRT{\operatorname{SRT}}
\newcommand\RT{\operatorname{RT}}

\newcommand\Row{\operatorname{Row}}

\newcommand\NASF{\operatorname{NASF}}

\newcommand\sh{\operatorname{sh}}
\renewcommand\neg{\operatorname{neg}}

\newcommand\Des{\operatorname{Des}}

\newcommand\maj{\operatorname{maj}}
\newcommand\leg{\operatorname{leg}}
\newcommand\arm{\operatorname{arm}}

\newcommand{\ZZ}{\mathbb{Z}}

\newcommand\SSYT{\operatorname{SSYT}}

\newcommand\wt{\operatorname{wt}}

\newcommand\Qbinom[3]{\genfrac{[}{]}{0pt}{}{#1}{#2}_{#3}}
\newcommand\qbinom[2]{\Qbinom{#1}{#2}{q}}

\newcommand\LLT{\operatorname{LLT}}

\newcommand\South{\operatorname{South}}

\definecolor{bluegray}{rgb}{0.4,0.6,0.8}

\definecolor{frenchblue}{rgb}{0.0, 0.45, 0.73}
\renewcommand\emph[1]{\textcolor{blue}{\it #1}}

\title{The \( e \)-expansion of integral form Macdonald polynomials}
\keywords{integral form Macdonald polynomials, \( e \)-positivity, chromatic symmetric functions}

\author{Jim Haglund}
\address{
Department of Mathematics, University of Pennsylvania, Philadelphia, PA 19104, USA}
\email{jhaglund@math.upenn.edu}

\author{Jang Soo Kim}
\address{Department of Mathematics,
Sungkyunkwan University, Suwon, Gyeonggi-do 16419, South Korea}
 \email{jangsookim@skku.edu}

\author{Jaeseong Oh}
\address{Department of Mathematics,
Sungkyunkwan University, Suwon, Gyeonggi-do 16419, South Korea}
\email{jaeseongoh@skku.edu}

\author{Meesue Yoo}
\thanks{The second author was supported by the National Research Foundation of Korea (NRF) grant funded by the Korea government RS-2025-00557835.
The third author was supported by the National Research Foundation of Korea (NRF) grant funded by the Ministry of Education (RS-2026-25469155).
The fourth author was supported by funding for the academic research program of Chungbuk National University in 2026 and NRF grant RS-2024-00344076.}
\address{
Department of Mathematics, Chungbuk National University, Cheongju 28644,
South Korea}
\email{meesueyoo@chungbuk.ac.kr}

\subjclass[2020]{Primary: 05E05; Secondary: 05A30, 33D15}

\begin{document}

\begin{abstract}
In this paper, we investigate the expansion of the integral form Macdonald polynomials $J_\mu (x;q,t)$ in the elementary symmetric function basis. Motivated by the connection between chromatic symmetric functions and LLT polynomials, we conjecture the $e$-positivity of a normalization of $J_\mu (x;q,t)$. We verify the conjecture when $\mu$ is either a hook or a two-column partition. In the hook case, we express the $e$-coefficients as weighted sums over compositions. In the two-column case, we derive explicit formulas for the $e$-coefficients.
\end{abstract}

\maketitle



\section{Introduction}
\label{sec:introduction}

Macdonald polynomials \( P_\mu(x;q,t) \) are a family of symmetric
functions introduced by Macdonald in \( 1988 \) \cite{Macdonald1988,
  Macdonald}. Macdonald polynomials generalize many well-known bases
of symmetric functions, including Schur functions and Hall--Littlewood
polynomials. Moreover, due to their deep connections with various
areas such as special function theory, representation theory,
algebraic geometry, and quantum mechanics, they have been the subject
of extensive study since their introduction. In the same work,
Macdonald also defined the integral form Macdonald polynomials
\( J_\mu(x;q,t) \) with coefficients in \( \ZZ[q,t] \):
\begin{equation}\label{eq:J=P}
  J_\mu (x;q,t)= P_\mu(x;q,t) \prod_{c\in\mu}(1-q^{\arm(c)}t^{\leg(c)+1}).
\end{equation}
See \Cref{sec:preliminaries} for the precise definitions.

In \( 2005 \), Haglund, Haiman, and Loehr \cite{HHL2005} made a major breakthrough in
Macdonald theory by discovering a combinatorial formula for the modified Macdonald
polynomials \( \tilde{H}_\mu (x;q,t) \), defined by
\begin{equation}\label{eq:H=J}
  \tilde{H}_\mu (x;q,t)=t^{n(\mu)}J_\mu \left[  \frac{X}{1-1/t};q,1/t \right],
\end{equation}
where \( X=x_1+x_2+\cdots \), the square brackets indicate
\emph{plethystic substitution}, and
\( n(\mu)=\sum_{i\geq 1}(i-1)\mu_i \). Haiman \cite{Haiman2001} proved
that \( \tilde{H}_\mu(x;q,t) \) is Schur positive, i.e., its Schur
coefficients belong to \( \ZZ_{\ge0}[q,t] \).

In \cite{HHL2005}, Haglund, Haiman, and Loehr expressed
\( \tilde{H}_\mu (x;q,t) \) as a positive linear combination of
\( \LLT \) polynomials. Introduced by Lascoux, Leclerc, and Thibon
\cite{LLT1997}, the \( \LLT \) polynomials \( \LLT_{\nu}(x;q) \) are
symmetric functions that are known to be Schur positive \cite{GH07}.
In general, LLT polynomials are indexed by tuples of skew shapes. The
special case in which each skew shape consists of a single cell gives
rise to the so-called \emph{unicellular LLT polynomials}, which can
also be indexed by Dyck paths.

The unicellular LLT polynomials are closely related to the
\emph{chromatic symmetric functions} \( \chi_\gamma (x;q) \)
introduced in \cite{SW2016} when \( \gamma \) is a Dyck path of
length \( 2n \) corresponding to the incomparability graph of a natural unit interval
order. More precisely, Haglund and Wilson \cite[Eq.~(6)]{Haglund2020}
observed that the result of Carlsson and Mellit
\cite[Proposition~3.5]{CM2018} is equivalent to
\begin{equation}\label{eq:X=LLT}
  \chi_\gamma(x;q)=\frac{\LLT_\gamma [(q-1)X;q]}{(q-1)^n}.
\end{equation}
The chromatic symmetric functions \( \chi_\gamma(x;q) \) have been shown to
expand positively in the elementary symmetric function basis (a
property we refer to as \emph{\( e \)-positivity} henceforth) when
\( q=1 \) \cite{Hikita2025, GMRWW2025}, and were conjectured by
Shareshian and Wachs \cite{SW2016} to have \( e \)-coefficients in
\( \ZZ_{\ge 0}[q] \).\footnote{OpenAI \cite{OpenAI2026} has claimed a proof
of this conjecture.}

We note that \eqref{eq:H=J} is equivalent to
\begin{equation}\label{eq:J=H}
      J_\mu\big(x; q, t^{-1}\big)
      = t^{-\sum_{i\ge1} i\mu_i}\, \tilde{H}_\mu\big[(t-1)X; q, t\big].
\end{equation}
The present work is motivated by the analogy between \eqref{eq:X=LLT}
and \eqref{eq:J=H}. On the right-hand sides, the unicellular LLT
polynomials \( \LLT_\gamma(x;q) \) and the modified Macdonald
polynomials \( \tilde{H}_\mu(x;q,t) \) are both Schur positive. In
view of this, the conjectured \( e \)-positivity of
\( \chi_\gamma(x;q) \) motivates the investigation of
\( e \)-positivity for suitably normalized specializations of integral
form Macdonald polynomials \( J_\mu\big(x; q, t^{-1}\big) \).

With regard to Schur positivity, Haglund conjectured
(see \cite[Eq.~(1)]{AHW2021}) that, for every partition \( \mu \)
and every positive integer \( k \), the symmetric function
\[
  \frac{J_\mu(x;q,q^k)}{(1-q)^{|\mu|}}
\]
has Schur coefficients in \( \ZZ_{\ge0}[q] \). Several special cases
have been proved in \cite{Aritra, Yoo2012, Yoo2015}. Furthermore,
Wilson~\cite{Wilson} conjectured that, for every negative integer
\( k \), the expression obtained by replacing \( (1-q) \) with
\( (q-1) \) and multiplying by a suitable power of \( q \) is also
Schur positive.

Building on this conjecture and the preceding discussion of the
\( e \)-positivity of \( J_\mu\big(x; q, t^{-1}\big) \) under an
appropriate normalization, we propose the following conjecture.

\begin{conj}\label{conj:epositivity}
  For any partition \( \mu \) and any positive integer
  \( k \), the following is \( e \)-positive:
\begin{equation}\label{eq:Schurpositivity}
  q^{k\sum_{i\ge1} i \mu_i}\frac{J_\mu (x;q,q^{-k})}{(q-1)^{|\mu|}}.
\end{equation}
\end{conj}

In this paper, we prove this conjecture when \( \mu \) is either a
hook or a two-column partition. To do this, we first consider the
\( e \)-expansion of the integral form Macdonald polynomial:
\[       
  J_\mu(x;q,t) = \sum_{\lambda\vdash |\mu|} c_{\mu,\lambda}(q,t) e_\lambda.
\]
When \( \mu \) is a hook, we express the coefficients
\( c_{\mu,\lambda}(q,t)=c_{\mu,\lambda} \) as weighted sums over compositions
(\Cref{thm:main}). As a consequence, we obtain a positive
\( e \)-expansion of \eqref{eq:Schurpositivity}
(\Cref{cor:hook-positivity}). When \( \mu \) has at most two columns,
we compute the coefficients explicitly as products
(\Cref{thm:two-column}), which implies the \( e \)-positivity of
\eqref{eq:Schurpositivity}. In addition, for an arbitrary partition
\( \mu \) of \( n \), we give a product formula for the coefficient
\( c_{\mu,(n)} \) of \( e_n \) (\Cref{thm:e_n}).

The rest of the paper is organized as follows. In
\Cref{sec:preliminaries} we recall basic definitions and known
results. In \Cref{sec:one-row-direct} we give an \( e \)-expansion
formula for \( J_\mu(x; q, t) \) when \( \mu \) is a one-row
partition. In \Cref{sec:hook} we extend this result to hooks and prove
\Cref{conj:epositivity} in this case. \Cref{sec:two-column} treats the
two-column case, and \Cref{sec:one-row} the \( e_n \)-coefficient.
Finally, in \Cref{sec:chrom-symm-funct}, we explain a connection to
\( \alpha \)-chromatic symmetric functions and propose a conjecture on
the \( q\to1 \) limit of \eqref{eq:Schurpositivity}.


\subsection*{Declaration on Generative AI}

The ChatGPT 5.5, 5.6, and 6 Pro models and Claude Opus~4.7 were used
for the following tasks:
\begin{itemize}
\item Simplifying the authors' original combinatorial proof of
  \Cref{thm:main} in \Cref{sec:combinatorial-proof}.
\item Finding a new recursive proof of \Cref{thm:main} in
  \Cref{sec:recursive-proof}.
\item Finding and proving \Cref{prop:one-row-positivity}
  and \Cref{cor:hook-positivity}.
\item Proving \Cref{thm:e_n}.
\item Proofreading the manuscript and conducting computational experiments with Sage.
\end{itemize}
All proofs obtained with the assistance of these AI tools have been
independently verified and rewritten by the authors.

\section{Preliminaries}
\label{sec:preliminaries}

In this section, we introduce the basic definitions and known results
that will be used throughout the paper.

\subsection{Basic definitions}

A \emph{composition} of \( n \) is a sequence
\( \alpha =(\alpha_1,\alpha_2,\dots,\alpha_l)\) of positive integers
satisfying \( \sum_i \alpha_i =n \), denoted by \( \alpha\vDash n \).
Each \( \alpha_i \) is called a \emph{part} of \( \alpha \), and the
number of parts is called the \emph{length} of \( \alpha \), denoted
by \( \ell(\alpha) \). The sum of the parts of \( \alpha \), denoted
by \( |\alpha| \), is called the \emph{size} of \( \alpha \). A
\emph{partition} is a composition such that the parts are weakly
decreasing. If \( \sum_i \lambda_i = n \), we say that \( \lambda \)
is a partition of \( n \) and write \( \lambda\vdash n \). For
notational convenience, we also write \( (0) \) for the empty partition
\( \emptyset \).

If a composition \( \alpha \) is equal to a rearrangement of a
partition \( \lambda \), we write \( \alpha\sim \lambda \). For
partitions \( \lambda, \mu\vdash n \), we define the \emph{dominance
  order} by \( \mu\le \lambda \) if
\[
  \mu_1+ \cdots +\mu_i \le \lambda_1 +\cdots + \lambda_i \quad \text{ for all } i\ge 1,
\]
where both
partitions are extended by zeros. We also write \( (0) \) for the empty composition
\( \emptyset \)

Given a partition \( \lambda=(\lambda_1,\dots,\lambda_k) \), its
\emph{diagram} is an array of unit squares, called \emph{cells}, with
\( \lambda_i \) cells in the \( i \)th row from the bottom,
left-justified. We often identify \( \lambda \) with its diagram. A
cell in row \( i \) and column \( j \) is denoted by \( (i,j) \). The
\emph{conjugate} \( \lambda' \) of \( \lambda \) is the partition
whose diagram is obtained from \( \lambda \) by reflecting across the
diagonal \( y=x \). The \emph{reading order} of \( \lambda \) is the
order of its cells obtained by reading the rows from top to bottom
and the cells in each row from left to right.

A partition of the form \( (r,1^s) \) with \( r\ge1 \) and \( s\ge0 \)
is called a \emph{hook}. Here, \( a^b \) means the sequence of \( b \)
\( a \)'s. Note that hooks include one-row partitions \( (n) \) and
one-column partitions \( (1^n)=(1,1^{n-1}) \).

Given a cell \( c\in \lambda \), the \emph{arm} (resp. \emph{coarm}) length of \( c=(i,j) \) is the number of cells
in row \( i \) to the right (resp. left) of \( c \) in \( \lambda \), denoted by \( \arm(c) \) (resp. \( \arm'(c) \)).
The \emph{leg} (resp. \emph{coleg}) length of \( c=(i,j) \) is the number of cells
in column \( j \) strictly above (resp. below) \( c \) in \( \lambda \), denoted by \( \leg(c) \) (resp. \( \leg'(c) \)). See \Cref{fig:1} for an example. 

\begin{figure}
  \centering
 \begin{tikzpicture}[scale=0.66]
    \tikzset{cell/.style={draw=black, line width=0.7pt}}
    \foreach \c/\r in {%
      4/1,                   
      1/2, 2/2, 3/2,         
      5/2, 6/2, 7/2, 8/2,    
      4/3, 4/4}{             
      \fill[gray!20] (\c-1,\r-1) rectangle (\c,\r);
    }
    \foreach \r/\L in {1/9, 2/8, 3/6, 4/6, 5/3}{
      \foreach \c in {1,...,\L}{
        \draw[cell] (\c-1,\r-1) rectangle (\c,\r);
      }
    }
    \node at (3.5,1.5) {$c$};
    \foreach \c in {1,2,3}{
      \node at (\c-0.5,1.5) {$\mathbf{a}'$};
    }
    \foreach \c in {5,6,7,8}{
      \node at (\c-0.5,1.5) {$\mathbf{a}$};
    }
    \foreach \r in {3,4}{
      \node at (3.5,\r-0.5) {$\mathbf{l}$};
    }
    \node at (3.5,0.5) {$\mathbf{l}'$};
  \end{tikzpicture}
  \caption{For the cell \( c=(2,4) \) in \( \lambda=(9,8,6,6,3) \),
    the arm, coarm, leg, and coleg lengths are 
\( \arm(c)=4 \),  \( \arm'(c)=3 \),
   \( \leg(c)=2 \) and  \( \leg'(c)=1 \), respectively.}
  \label{fig:1}
\end{figure}

For a partition \( \lambda \), a \emph{semistandard Young tableau} of
shape \( \lambda \) is a filling of \( \lambda \) with positive
integers such that the entries are weakly increasing from left to
right in each row and strictly increasing from bottom to top in each
column. We use \( \SSYT(\lambda) \) to denote the set of semistandard
Young tableaux of shape \( \lambda \).

For a nonnegative integer \( r \) and a partition \( \mu \), we use
the \emph{\( q \)-Pochhammer symbols}
\[
  (a;q)_r = \prod_{i=0}^{r-1}(1-aq^i),
  \qquad
  (a;q)_\mu = \prod_{i=1}^{\ell(\mu)}(a;q)_{\mu_i},
\]
where \( (a;q)_0=1 \). For integers \( n,k \) with
\( 0\le k\le n \), the \emph{\( q \)-binomial coefficient} is defined by
\[
  \qbinom{n}{k}
  =\frac{(q;q)_n}{(q;q)_k(q;q)_{n-k}}.
\]
We also define
\[
  [n]_q=1+q + \cdots + q^{n-1}, \qquad [n]_q!=\prod_{i=1}^n[i]_q.
\]

\subsection{Symmetric functions}

Let \( \Lambda_F \) denote the \( F \)-algebra of symmetric functions
in infinitely many variables \( x=(x_1,x_2,\dots) \) with coefficients in
\( F=\mathbb{Q}(q,t) \). For a partition
\( \lambda=(\lambda_1,\dots,\lambda_l) \), the \emph{monomial
  symmetric function} \( m_{\lambda} \) is defined by
\( m_{\lambda} = \sum x^{\alpha}, \) where the sum is over all
sequences \( \alpha=(\alpha_1,\alpha_2,\dots) \) obtained by
rearranging \( (\lambda_1,\dots,\lambda_l,0,0,\dots) \), and
\( x^\alpha=x_1^{\alpha_1}x_2^{\alpha_2}\cdots \). The
\emph{elementary symmetric function} \( e_\lambda \) is defined by
\( e_\lambda =\prod_i e_{\lambda_i}\), where
\( e_n=\sum_{ i_1<\cdots <i_n} x_{i_1}\cdots x_{i_n}\). The
\emph{complete homogeneous symmetric function} \( h_\lambda \) is
defined by \( h_\lambda =\prod_{i} h_{\lambda_i} \), where
\( h_n =\sum_{i_1\le \cdots \le i_n}x_{i_1}\cdots x_{i_n} \). The
\emph{power sum symmetric function} \( p_\lambda \) is defined by
\( p_\lambda =\prod_i p_{\lambda_i} \), where
\( p_n = \sum_i x_i ^n \). We also define \( e_0=h_0=p_0=1 \). The
\emph{Schur function} \( s_\lambda \) is defined combinatorially by
\[
  s_\lambda = \sum_{T\in \SSYT(\lambda )} x^T,
\]
where \( x^T=\prod_i x_i^{a_i} \) with \( a_i \) the number of
occurrences of \( i \) in \( T \).

We note that each of \( \{m_\lambda\} \), \( \{e_\lambda\} \), \( \{h_\lambda\} \), \( \{ p_\lambda\} \), and \( \{s_\lambda\} \) forms
a basis of the space of symmetric functions. 
With a slight abuse of notation, for any of the bases \( \{g_\lambda\} \) 
listed above, we write \( [g_\lambda]f \) for the coefficient of \( g_\lambda \)
in the expansion of a symmetric function \( f\in \Lambda_F \) in that basis.

Let \( E=E(t_1, t_2, \dots) \) be a formal Laurent series in
\( t_1, t_2,\dots \) with rational coefficients. For
\( f=\sum_{\lambda}c_\lambda p_\lambda \in \Lambda_F \), the
\emph{plethystic substitution} \( f[E] \) is defined by
\[
  f[E]:=\sum_{\lambda}c_\lambda p_\lambda [E],
\]
where \( p_\lambda [E] = \prod_{i} p_{\lambda_i}[E] \) and
\( p_k[E] = E\left(t_1 ^k, t_2 ^k, \dots\right)\). We refer the reader
to \cite[Ch.~I, \S8]{Macdonald} for more details on plethysm.

\subsection{Macdonald polynomials}

The \emph{Macdonald polynomials} \( P_\mu=P_\mu(x;q,t) \) are the
unique symmetric functions satisfying
\[
  P_\mu=m_\mu+\sum_{\lambda<\mu}u_{\mu\lambda}(q,t)m_\lambda,
  \qquad
  \langle P_\lambda,P_\mu\rangle_{q,t}=0\quad\text{if }\lambda\ne\mu,
\]
for suitable coefficients \( u_{\mu\lambda}(q,t)\in F \); see
\cite{Macdonald1988,Macdonald}. Here,
\( \langle \cdot ,\cdot \rangle_{q,t} \) is the inner product on
\( \Lambda_F \) determined by
\begin{equation}\label{eq:z_la}
  \langle p_\lambda,p_\mu\rangle_{q,t}
  =\delta_{\lambda\mu}z_\lambda
    \prod_{i=1}^{\ell(\lambda)}\frac{1-q^{\lambda_i}}{1-t^{\lambda_i}},
  \qquad
  z_\lambda=\prod_{r\ge1}r^{m_r(\lambda)}m_r(\lambda)!,
\end{equation}
where \( m_r(\lambda) \) is the multiplicity of \( r \) in
\( \lambda \). The \emph{integral form Macdonald polynomials}
\( J_\mu(x;q,t) \) and the \emph{modified Macdonald polynomials}
\( \tilde{H}_\mu(x;q,t) \) are defined by
\begin{align}
  \notag
  J_\mu(x;q,t)
  &= P_\mu(x;q,t) \prod_{c\in\mu}(1-q^{\arm(c)}t^{\leg(c)+1}),\\
  \label{eq:1}
  \tilde{H}_\mu[X;q,t]
  &=t^{n(\mu)}J_\mu\left[\frac{X}{1-t^{-1}};q,t^{-1}\right],
\end{align}
where \( n(\mu)=\sum_i (i-1)\mu_i \). In \cite{HHL2005}, Haglund,
Haiman, and Loehr discovered a remarkable combinatorial formula for
\( \tilde{H}_\mu(x;q,t) \). To state their result, we need some
definitions.

For a partition \( \mu\vdash n \), a \emph{filling} \( \sigma:\mu\rightarrow \ZZ^+ \) is an
assignment of positive integers to the cells in the diagram of \( \mu \).
A filling is \emph{standard} if the entries are precisely \( 1,2,\dots,|\mu| \). 
The monomial \( x^\sigma \) associated with the filling \( \sigma \) is defined by 
\[
  x^\sigma = \prod_{c\in \mu}x_{\sigma(c)}.
\]

For a cell \( c=(i,j)\in \mu \), let \( \South(i,j)=(i-1, j) \). Given a filling
\( \sigma:\mu\rightarrow \ZZ^+ \), the \emph{descent set}
\( \Des(\sigma) \)
and the \emph{major index} \( \maj (\sigma) \)
are defined by
\begin{align*}
  \Des(\sigma)
  &=\{ c\in \mu :\South(c)\in\mu,\,  \sigma(c)> \sigma(\South (c))\},\\
  \maj (\sigma)
  &=\sum_{c\in\Des(\sigma)}(\leg(c)+1).
\end{align*}

A \emph{triple} of \( \mu \) is a set \( \{u,v,w\} \) of three cells
with \( u,v\in\mu \) such that \( v \) is in the same row as \( u \)
and strictly to its right, and \( w=\South(u) \), which may be below
the first row. Let \( \sigma \) be a filling of \( \mu \). We orient a
triple by starting at the cell whose entry is smallest, proceeding to
the cell with the next-largest entry, and ending at the cell with the
largest entry. If two cells have the same entry, the entry of the cell
that appears earlier in the reading order of \( \mu \) is considered
smaller. The triple is an \emph{inversion triple} if this circular
motion is counterclockwise, and a \emph{coinversion triple} if it is
clockwise. We denote by \( \inv(\sigma) \)
(resp. \( \coinv(\sigma) \)) the number of inversion triples
(resp. coinversion triples) of \( \sigma \). Here, we define
\( \sigma(w)=\infty \) if \( w \) is below the first row. Note that
\[
  \coinv(\sigma) = n(\mu') - \inv(\sigma).
\]

\begin{thm}[{\cite[Theorem 2.2]{HHL2005}}]\label{thm:1}
  For a partition \( \mu \), the modified Macdonald polynomial \( \tilde{H}_\mu(x;q,t) \)
  is given by
  \begin{equation}\label{eq:11}
\tilde{H}_\mu (x;q,t)=\sum_{\sigma:\mu\rightarrow \ZZ^+}x^\sigma q^{\inv(\sigma)}t^{\maj(\sigma)}.
  \end{equation}
 \end{thm}

 By \eqref{eq:1}, we have 
 \begin{equation}\label{eqn:JandH}
J_\mu (x;q,t)   = t^{n(\mu)}\tilde{H}_\mu [X(1-t);q,t^{-1}] 
               = t^{n(\mu)+n}\tilde{H}_{\mu'}[X(t^{-1} -1);t^{-1}, q].
 \end{equation}
 To state a combinatorial interpretation for \( J_\mu (x;q,t) \), we
 introduce the necessary definitions.

 We say that a cell \( u\in\mu \) \emph{attacks} the cells in the same
 row and strictly to its right, and the cells in the row below and
 strictly to its left. A filling \( \sigma:\mu\rightarrow \ZZ^+ \) is
 \emph{nonattacking} if there are no cells \( u, v\in \mu \) such that
 \( u \) attacks \( v \) and \( \sigma(u)=\sigma(v)\).

 Let \( \mathcal{A}_+=\{1,2,\dots\} \),
 \( \mathcal{A}_- =\{\overline{1},\overline{2},\dots\} \), and
 \( \mathcal{A}_{\pm}=\mathcal{A}_+\cup \mathcal{A}_- \). We write
 \( |i|=i \) and \( |\overline{i}|=i \). A \emph{superfilling} is an
 assignment \( \sigma:\mu\rightarrow \mathcal{A}_\pm \) of elements in
 \( \mathcal{A}_\pm \) to the cells in \( \mu \). We denote by
 \( |\sigma| \) the filling obtained from \( \sigma \) by replacing
 each entry \( \overline{i} \) with \( i \). We say that a
 superfilling \( \sigma \) is \emph{nonattacking} if \( |\sigma| \) is
 nonattacking. We denote by \( \neg(\sigma) \) the number of entries
 of \( \sigma \) in \( \mathcal{A}_- \). Let \( \NASF(\mu) \) denote
 the set of nonattacking superfillings of \( \mu \).

 Given a total order \( \prec \) on \( \mathcal{A}_\pm \), the
 \emph{standardization} of \( \sigma \) with respect to \( \prec \) is
 the bijective filling
 \( \sigma^{\mathrm{st}}_{\prec}:\mu\to\{1,\dots,n\} \) obtained by
 relabeling the cells in increasing order of their entries, where
 ties between cells holding the same letter are broken by reading
 order: a letter in \( \mathcal{A}_+ \) read earlier is regarded as
 smaller, and a letter in \( \mathcal{A}_- \) read earlier is regarded
 as larger.

Haglund, Haiman, and Loehr \cite[Eq.~(63)]{HHL2005} found a
combinatorial interpretation for \( J_\mu(x ; q, t) \). We will use
the following reformulation from \cite[(A.30)]{Haglund}.

\begin{thm}\label{thm:NASF}
  We have
\begin{equation}\label{eq:HHL}
  J_\mu(x ; q, t)=\sum_{\sigma\in \NASF(\mu')} x^{|\sigma|} 
    q^{\maj(\sigma)} t^{\coinv (\sigma)}(-t)^{\neg(\sigma)},
\end{equation}
where \( \maj (\sigma) = \maj(\sigma^{\mathrm{st}}_\prec) \) and
\( \coinv (\sigma)= \coinv (\sigma^{\mathrm{st}}_\prec) \), with
standardization taken with respect to the ordering
\begin{equation}\label{eq:A16}
1\prec \overline{1}\prec 2\prec \overline{2}\prec \cdots\prec n\prec \overline{n}\prec  \cdots.
\end{equation}
\end{thm}

\section{One-row case}\label{sec:one-row-direct}

In this section, we give an explicit combinatorial formula for the
\( e \)-coefficients of \( J_{(n)}(x;q,t) \). We give two proofs of
this formula. We first define a weight of a composition with respect
to a one-row partition.

\begin{defn}
  For a composition \( \alpha=(\alpha_1,\dots,\alpha_k) \) of \( n \),
  let 
\begin{equation}\label{eq:wt(i,a)}
    \wt(i,\alpha) =
    \begin{cases}
      1-q^{i-1}t^{\alpha_j} & \mbox{if \( i= \alpha_1 + \cdots + \alpha_{j-1}+1 \) for some \( 1\le j\le k \)},\\
      q^{i-1}-1 & \mbox{otherwise.}
    \end{cases}
\end{equation}
For the one-row partition \( (n) \), define
\begin{equation}\label{eq:wt-one-row}
  \wt_{(n)}(\alpha) = \prod_{i=1}^{n} \wt(i,\alpha).
\end{equation}
\end{defn}
It is helpful to visualize the weight \( \wt_{(n)}(\alpha) \) as in the
next example.

\begin{exam}\label{ex:one-row-weight}
  Let \( n=10 \) and \( \alpha=(2,1,3,4) \). Draw a row divided into
  blocks of sizes \( \alpha_1,\alpha_2,\alpha_3,\alpha_4 \), that is,
  \( 2,1,3,4 \), from left to right, as illustrated in
  \Cref{fig:one-row-weight-example}. The first cells of the blocks have
  indices \( i=1,3,4,7 \). We have
  \[
    \begin{aligned}
      \wt_{(10)}(2,1,3,4)
      ={}&(1-t^2)(q-1)(1-q^2t)(1-q^3t^3)\\
      &{}\times(q^4-1)(q^5-1)(1-q^6t^4)(q^7-1)(q^8-1)(q^9-1).
    \end{aligned}
  \]
\end{exam}

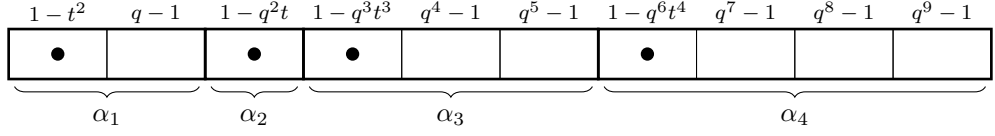
\begin{figure}
  \centering
  \begin{tikzpicture}[x=1.3cm,y=0.66cm]
    \foreach \start/\length/\blocklabel in
      {0/2/{\alpha_1},2/1/{\alpha_2},
       3/3/{\alpha_3},6/4/{\alpha_4}}{
      \foreach \c in {1,...,\length}{
        \draw[line width=0.35pt]
          (\start+\c-1,0) rectangle (\start+\c,1);
      }
      \draw[line width=1pt] (\start,0) rectangle (\start+\length,1);
      \fill[black] (\start+0.5,0.5) circle[radius=2.5pt];
      \draw[decorate,decoration={brace,mirror,amplitude=4pt}]
        (\start+0.06,-0.15) -- (\start+\length-0.06,-0.15)
        node[midway,below=5pt] {\(\blocklabel\)};
    }
    \foreach \cell/\cellweight in
      {1/{1-t^2},2/{q-1},3/{1-q^2t},4/{1-q^3t^3},
       5/{q^4-1},6/{q^5-1},7/{1-q^6t^4},8/{q^7-1},
       9/{q^8-1},10/{q^9-1}}{
      \node[anchor=south,font=\footnotesize,inner sep=2pt]
        at (\cell-0.5,1) {\(\cellweight\)};
    }
  \end{tikzpicture}
  \caption{For \( \alpha=(2,1,3,4) \), a row divided into
    blocks of sizes \( \alpha_1,\alpha_2,\alpha_3,\alpha_4 \)
    from left to right.
    A dot marks the first cell of each block, and the weight
    \( \wt(i,\alpha) \) is written above the \( i \)th cell.}
  \label{fig:one-row-weight-example}
\end{figure}

\begin{thm}\label{thm:one-row}\label{thm:one-row-direct}
  For \( n\ge0 \), we have
  \[
    J_{(n)}(x;q,t) = \sum_{\lambda\vdash n}
    \left( \sum_{\alpha\sim \lambda} \wt_{(n)}(\alpha) \right) e_\lambda.
  \]
\end{thm}

To prove this theorem, it is convenient to use the following notation.
For a composition \( \alpha \), write
\( e_\alpha=\prod_j e_{\alpha_j} \) and let
\begin{equation}\label{eq:Y-one-row}
  Y_{(n)}=Y_{(n)}(x;q,t)
  =\sum_{\alpha\vDash n}\wt_{(n)}(\alpha)e_\alpha.
\end{equation}
Then the identity in \Cref{thm:one-row-direct} is equivalent to
\begin{equation}\label{eq:J=Y}
  J_{(n)}(x;q,t) = Y_{(n)}(x;q,t).
\end{equation}
We provide two proofs of this identity in the next two subsections.
The first proof is combinatorial and the second proof is recursive.

\subsection{Combinatorial proof}
\label{sec:combinatorial-proof}

In this subsection, we prove \eqref{eq:J=Y} by comparing monomial
coefficients on both sides using the Haglund--Haiman--Loehr formula
\eqref{eq:HHL}.

For the one-row partition \( \mu=(n) \), we can identify each
superfilling of \( \mu'=(1^n) \) with its reading word. Using this
identification, we first restate \eqref{eq:HHL} in terms of words on
\( \mathcal{A}_{\pm} \). We denote by \( W_n(\mathcal{A}_{\pm}) \) the
set of words of length \( n \) on \( \mathcal{A}_{\pm} \). We define
\( W_n(\mathcal{A}_{+}) \) and \( W_n(\mathcal{A}_{-}) \) similarly.
For \( \pi=\pi_1 \cdots \pi_n \in W_n(\mathcal{A}_{\pm}) \), we define
\( |\pi|=|\pi_1| \cdots |\pi_n| \in W_n(\mathcal{A}_{+}) \). For
\( \pi=\pi_1 \cdots \pi_n \in W_n(\mathcal{A}_{+}) \), we write
\( x^\pi = x_{\pi_1} \cdots x_{\pi_n} \).

For \( \pi\in W_n(\mathcal{A}_{\pm}) \), let \( \neg(\pi) \) denote the
number of letters of \( \pi \) in \( \mathcal{A}_- \). We define
\[
  \Des(\pi)=\{1\le i\le n-1:\pi_i\succ\pi_{i+1}
    \text{ or }\pi_i=\pi_{i+1}\in\mathcal A_-\},
  \qquad \maj(\pi)=\sum_{i\in \Des(\pi)}i,
\]
where \( \prec \) is the ordering of \( \mathcal{A}_{\pm} \)
in \eqref{eq:A16}.
For later use, we also define
\[
  \Des^*(\pi)=\{1\le i\le n-1:\pi_i\succ^*\pi_{i+1}
    \text{ or }\pi_i=\pi_{i+1}\in\mathcal A_-\},
  \qquad \maj^*(\pi)=\sum_{i\in \Des^*(\pi)}i,
\]
where \( \prec^* \) is the ordering of \( \mathcal{A}_{\pm} \)
given by
\begin{equation}\label{eq:3}
  \overline{1}\prec^* \overline{2}\prec^*
  \cdots\prec^* 1\prec^* 2 \prec^* \cdots.
\end{equation}

Now consider \( \sigma\in \NASF((1^n)) \) and let \( \pi \) be the
reading word of \( \sigma \). Then \( \maj(\sigma)=\maj(\pi) \) and
\( \neg(\sigma)=\neg(\pi) \). Since the one-column partition \( (1^n) \)
contains no attacking pairs, the map \( \sigma\mapsto\pi \) is a
bijection from \( \NASF((1^n)) \) to \( W_n(\mathcal{A}_{\pm}) \). Moreover,
since \( (1^n) \) has no triples, we always have
\( \coinv(\sigma)=0 \). Therefore, \eqref{eq:HHL} gives
\begin{equation}\label{eq:one-row-hhl-local}
  J_{(n)}(x;q,t)
  =\sum_{\pi\in W_n(\mathcal{A}_{\pm})}x^{|\pi|} q^{\maj(\pi)}(-t)^{\neg(\pi)}.
\end{equation}

\begin{lem}\label{lem:maj*}
  We have
 \[
    Y_{(n)}(x;q,t)
  =\sum_{\pi\in W_n(\mathcal{A}_{\pm})}x^{|\pi|} q^{\maj^*(\pi)}(-t)^{\neg(\pi)}.
  \]
\end{lem}

\begin{proof}
  For a composition \( \alpha = (\alpha_1,\dots,\alpha_k)\vDash n \)
  and \( 1\le i\le k \), we define the \( i \)th \emph{block} of
  \( \alpha \) to be
  \[
    B_i(\alpha) = \{\alpha_1 + \cdots + \alpha_{i-1}+1,
    \alpha_1 + \cdots + \alpha_{i-1}+2,\dots,
    \alpha_1 + \cdots + \alpha_i\}.
  \]
  By definition, we have
  \begin{equation}\label{eq:5}
    Y_{(n)}(x;q,t) =\sum_{\alpha\vDash n}\wt_{(n)}(\alpha)e_\alpha
    = \sum_{(\alpha,\pi)\in C_n} \wt_{(n)}(\alpha) x^\pi ,
  \end{equation}
  where \( C_n \) is the set of pairs \( (\alpha,\pi) \) such that
  \( \alpha = (\alpha_1,\dots,\alpha_k)\vDash n \),
  \( \pi=\pi_1 \cdots \pi_n\in W_n(\mathcal{A}_{+}) \), and
  the subword of \( \pi \) at positions in \( B_i(\alpha) \) is
  strictly decreasing for each \( 1\le i\le k \). 

  Now we reformulate \eqref{eq:5} using words on
  \( \mathcal{A}_{\pm} \). Writing
  \( B_i(\alpha) = \{r_i,r_i+1,\dots,s_i\} \), we have
  \begin{align*}
    \wt_{(n)}(\alpha)
    &= \prod_{i=1}^{k} \left((1-q^{r_i-1}t^{\alpha_i}) \prod_{j=r_i+1}^{s_i}(q^{j-1}-1)\right)\\
    &= \prod_{i=1}^{k} \left(
\prod_{j=r_i+1}^{s_i}(q^{j-1}-1) + q^{r_i-1}(-t)^{\alpha_i} \prod_{j=r_i+1}^{s_i}(1-q^{j-1})\right).
  \end{align*}
  Choosing the first or second term in the \( i \)th factor can be
  interpreted as leaving the subword \( \pi_{r_i} \cdots \pi_{s_i} \)
  unchanged or replacing it with
  \( \overline{\pi_{s_i}} \cdots \overline{\pi_{r_i}} \),
  respectively. Using this idea, we define \( D_n \) to be the set of
  pairs \( (\alpha,\pi) \) such that
  \( \alpha = (\alpha_1,\dots,\alpha_k)\vDash n \),
  \( \pi=\pi_1 \cdots \pi_n \in W_n(\mathcal{A}_{\pm}) \), and for each
  \( 1\le i\le k \), the subword of \( \pi \) at positions in
  \( B_i(\alpha) \) is either a strictly decreasing word on
  \( \mathcal{A}_{+} \) or a strictly increasing word on
  \( \mathcal{A}_{-} \) with respect to the order \( \prec \) in
  \eqref{eq:A16}. For \( (\alpha,\pi)\in D_n \), we define
  \[
    \wt(\alpha,\pi) =  \prod_{i=1}^{n} \wt(i,\alpha,\pi),
  \]
  where
 \[
    \wt(i,\alpha,\pi) =
    \begin{cases}
      1 & \mbox{if \( i= \min(B_j(\alpha)) \) for some \( 1\le j\le \ell(\alpha) \)
                              and \( \pi_i\in \mathcal{A}_+ \)},\\
      q^{i-1}(-t)^{\alpha_j} & \mbox{if \( i= \min(B_j(\alpha)) \) for some \( 1\le j\le \ell(\alpha) \)
                              and \( \pi_i\in \mathcal{A}_- \)},\\
      q^{i-1}-1 & \mbox{if \( i\ne \min(B_j(\alpha)) \) for all \( 1\le j\le \ell(\alpha) \)
                              and \( \pi_i\in \mathcal{A}_+ \)},\\
      1-q^{i-1} & \mbox{if \( i\ne \min(B_j(\alpha)) \) for all \( 1\le j\le \ell(\alpha) \)
                              and \( \pi_i\in \mathcal{A}_- \)}.
    \end{cases}
\]
Then \eqref{eq:5} can be rewritten as
\[
     Y_{(n)}(x;q,t)
    = \sum_{(\alpha,\pi)\in D_n} \wt(\alpha,\pi) x^{|\pi|}
    = \sum_{\pi\in W_n(\mathcal{A}_{\pm})}x^{|\pi|}
    \sum_{\alpha:(\alpha,\pi)\in D_n} \wt(\alpha,\pi) ,
\]
where \( x^{|\pi|} = x_{|\pi_1|} \cdots x_{|\pi_n|} \).
Therefore, it suffices to show that for each
\( \pi\in W_n(\mathcal{A}_{\pm}) \),
\begin{equation}\label{eq:6}
  \sum_{\alpha\vDash n} \wt(\alpha,\pi)
  = q^{\maj^*(\pi)}(-t)^{\neg(\pi)},
\end{equation}
where we extend the definition of \( \wt(\alpha,\pi) \) to all
compositions \( \alpha\vDash n \) by setting
\( \wt(\alpha,\pi) = 0 \) unless \( (\alpha,\pi)\in D_n \).

Identifying a composition \( \alpha = (\alpha_1,\dots,\alpha_k) \)
with its blocks \( B_1(\alpha),\dots,B_k(\alpha) \), we can construct
a composition of \( n \) by choosing whether \( i \) and \( i+1 \)
belong to the same block for each \( 1\le i\le n-1 \). We say that
there is a \emph{cut} between \( i \) and \( i+1 \) if they belong to
different blocks. Note that the cuts can be chosen independently, so
there are \( 2^{n-1} \) compositions of \( n \) in total.

Now we fix \( \pi\in W_n(\mathcal{A}_{\pm}) \). We define the weights
\( c_i(\pi) \) and \( d_i(\pi) \) for cutting and not cutting,
respectively, between \( i \) and \( i+1 \) according to the
signs of \( \pi_i \) and \( \pi_{i+1} \), as follows:
  \[
    \begin{array}{c|c|c}
      \text{signs of }(\pi_i,\pi_{i+1}) & c_i(\pi) & d_i(\pi) \\ \hline
      (+,+)&1&(q^i-1)\chi(|\pi_i|>|\pi_{i+1}|)\\
      (-,-)&q^i&(1-q^i)\chi(|\pi_i|<|\pi_{i+1}|)\\
      (+,-)&q^i&0\\
      (-,+)&1&0
    \end{array}
  \]
  Here \( \chi(P)=1 \) if \( P \) holds and \( \chi(P)=0 \) otherwise.
  Then, by the construction and the definition of
  \( \wt(\alpha,\pi) \), we have
  \begin{equation}\label{eq:7}
    \sum_{\alpha\vDash n} \wt(\alpha,\pi) =
    (-t)^{\neg(\pi)} \prod_{i=1}^{n-1} (c_i(\pi) + d_i(\pi)).
  \end{equation}
  On the other hand, for each \( 1\le i\le n-1 \), we have
  \[
    c_i(\pi) + d_i(\pi) =
    \begin{cases}
     q^i & \mbox{if \( \pi_i,\pi_{i+1}\in \mathcal{A}_+ \) and \( \pi_i>\pi_{i+1} \)},\\
     1 & \mbox{if \( \pi_i,\pi_{i+1}\in \mathcal{A}_+ \) and \( \pi_i\le \pi_{i+1} \)},\\
     1 & \mbox{if \( \pi_i,\pi_{i+1}\in \mathcal{A}_- \) and \( |\pi_i|< |\pi_{i+1}| \)},\\
     q^i & \mbox{if \( \pi_i,\pi_{i+1}\in \mathcal{A}_- \) and \( |\pi_i|\ge |\pi_{i+1}| \)},\\
     q^i & \mbox{if \( \pi_i\in \mathcal{A}_+ \) and \( \pi_{i+1}\in \mathcal{A}_- \),}\\
     1 & \mbox{if \( \pi_i\in \mathcal{A}_- \) and \( \pi_{i+1}\in \mathcal{A}_+ \).}
    \end{cases}
  \]
  This implies that \( c_i(\pi) + d_i(\pi) \) is equal to \( q^i \) if
  \( i\in\Des^*(\pi) \) and \( 1 \) otherwise.
  Therefore, \eqref{eq:6} follows from \eqref{eq:7}.
\end{proof}

For a composition \( \beta=(\beta_1,\dots,\beta_d)\vDash n \), let
\( B(\beta) \) be the set of words
\( \pi \in W_n(\mathcal{A}_{\pm}) \) in which \( i \) and
  \( \overline i \) occur \( \beta_i \) times in total for each
  \( 1\le i\le d \). Equivalently, \( \pi\in B(\beta) \) if and only if
  \( x^{|\pi|} = x^\beta \).

\begin{lem}\label{lem:maj*=maj}
  For \( \beta\vDash n \), we have
 \[
    \sum_{\pi\in B(\beta)}q^{\maj(\pi)}(-t)^{\neg(\pi)}
    =\sum_{\pi\in B(\beta)}q^{\maj^*(\pi)}(-t)^{\neg(\pi)}.
  \]
\end{lem}

\begin{proof}
  Let \( \beta=(\beta_1,\dots,\beta_d)\vDash n \). For a word
  \( \pi^+ \) on \( \mathcal{A}_+ \) and a word \( \pi^- \) on
  \( \mathcal{A}_- \), a \emph{shuffle} of \( \pi^+ \) and \( \pi^- \)
  is a word \( \pi \) on \( \mathcal{A}_{\pm} \) such that the subword
  of \( \pi \) consisting of letters in \( \mathcal{A}_+ \)
  (resp.~\( \mathcal{A}_- \)) is equal to \( \pi^+ \)
  (resp.~\( \pi^- \)). Let \( \sh(\pi^+,\pi^-) \) denote the set of
  such shuffles. Let \( V(\beta) \) be the set of pairs
  \( (\pi^+,\pi^-) \) of words on \( \mathcal{A}_+ \) and
  \( \mathcal{A}_- \), respectively, such that, for each
  \( 1\le i\le d \), the total number of occurrences of \( i \) in
  \( \pi^+ \) and \( \overline{i} \) in \( \pi^- \) is \( \beta_i \). Then we
  have
  \begin{equation}\label{eq:2}
    \sum_{\pi\in B(\beta)}q^{\maj(\pi)}(-t)^{\neg(\pi)}
    =\sum_{(\pi^+,\pi^-)\in V(\beta)}(-t)^{k}
      \sum_{\pi\in\sh(\pi^+,\pi^-)}q^{\maj(\pi)},
  \end{equation}
  where \( k \) is the total number of letters in \( \pi^- \).
  Since the identity \eqref{eq:2} with each \( \maj \) replaced by
  \( \maj^* \) also holds, it suffices to show that for
  \( (\pi^+,\pi^-)\in V(\beta) \), we have
  \begin{equation}\label{eq:4}
    \sum_{\pi\in\sh(\pi^+,\pi^-)}q^{\maj(\pi)}
    = \sum_{\pi\in\sh(\pi^+,\pi^-)}q^{\maj^*(\pi)}.
  \end{equation}

  Let \( \sigma^+ \) and \( \sigma^- \) (resp.~\( \tau^+ \) and
  \( \tau^- \)) be the words such that the concatenation
  \( \sigma^+\sigma^- \) (resp.~\( \tau^+\tau^- \)) is the
  standardization of \( \pi^+\pi^- \) with respect to the order
  \( \prec \) in \eqref{eq:A16} (resp.~\( \prec^* \) in \eqref{eq:3}).
  Then by the definitions of \( \maj \) and \( \maj^* \) and by the
  well-known fact \cite[Ch.~3, Exercise~162(b)]{EC1}, we have
  \begin{align*}
    \sum_{\pi\in\sh(\pi^+,\pi^-)}q^{\maj(\pi)}
    &= \sum_{\sigma\in\sh(\sigma^+,\sigma^-)}q^{\maj(\sigma)}
      =q^{\maj(\sigma^+)+\maj(\sigma^-)}\qbinom{n}{k},\\
    \sum_{\pi\in\sh(\pi^+,\pi^-)}q^{\maj^*(\pi)}
    &= \sum_{\tau\in\sh(\tau^+,\tau^-)}q^{\maj(\tau)}
      =q^{\maj(\tau^+)+\maj(\tau^-)}\qbinom{n}{k}.
  \end{align*}
  Since the orderings \eqref{eq:A16} and \eqref{eq:3} agree on
  \( \mathcal{A}_+ \) and on \( \mathcal{A}_- \), we
  have \( \maj(\sigma^+) = \maj(\tau^+) \) and
  \( \maj(\sigma^-) = \maj(\tau^-) \). Therefore, we obtain
  \eqref{eq:4}, which completes the proof.
\end{proof}

Now we prove \eqref{eq:J=Y}.

\begin{proof}[First proof of \eqref{eq:J=Y}]
  For every \( \beta\vDash n \), \eqref{eq:one-row-hhl-local}
  and \Cref{lem:maj*=maj,lem:maj*} give
  \[
    [x^\beta]J_{(n)}
    =\sum_{\pi\in B(\beta)}q^{\maj(\pi)}(-t)^{\neg(\pi)}
    =\sum_{\pi\in B(\beta)}q^{\maj^*(\pi)}(-t)^{\neg(\pi)}
    =[x^\beta]Y_{(n)}.
  \]
  Since both sides of \eqref{eq:J=Y} are symmetric functions, the
  identity follows.
\end{proof}

\subsection{Recursive proof}
\label{sec:recursive-proof}

In this subsection, we give a recursive proof of \eqref{eq:J=Y}. We
first recall the generating function of \( J_{(n)}(x;q,t) \).

\begin{lem}\label{lem:one-row-generating}
  We have
  \[
    \sum_{n\ge0}\frac{J_{(n)}(x;q,t)}{(q;q)_n}u^n
    =\prod_{i\ge1}A(x_i u),
    \qquad
    A(v):=\sum_{m\ge0}\frac{(t;q)_m}{(q;q)_m}v^m.
  \]
\end{lem}

\begin{proof}
  Let \( g_n(x;q,t) \) be the symmetric functions given by
  \cite[Ch.~VI, (2.8), p.~311]{Macdonald}
  \[
    \sum_{n\ge0}g_n(x;q,t)u^n
    =\prod_{i\ge1}\frac{(tx_i u;q)_\infty}{(x_i u;q)_\infty}.
  \]
  By \cite[Ch.~VI, (4.9), p.~323]{Macdonald} and \cite[Ch.~VI, (8.1)
  and (8.3), p.~352]{Macdonald}, we have
  \[
    J_{(n)}(x;q,t)=(t;q)_nP_{(n)}(x;q,t)
    =(q;q)_ng_n(x;q,t).
  \]
  The \( q \)-binomial identity \cite[Ch.~VI, \S2, Example~1,
  p.~314]{Macdonald} gives
  \[
    A(v)=\sum_{m\ge0}\frac{(t;q)_m}{(q;q)_m}v^m
    =\frac{(tv;q)_\infty}{(v;q)_\infty}.
  \]
  Combining these identities, we obtain
  \[
    \sum_{n\ge0}\frac{J_{(n)}(x;q,t)}{(q;q)_n}u^n
    =\sum_{n\ge0}g_n(x;q,t)u^n
    =\prod_{i\ge1}A(x_i u),
  \]
  as desired.
\end{proof}

\begin{proof}[Second proof of \eqref{eq:J=Y}]
  We first derive a recurrence for \( Y_{(n)} \) by considering the
  last part of each composition of \( n\ge1 \). For \( 1\le k\le n \),
  write \( \alpha=(\beta,k) \) with \( \beta\vDash n-k \). By
  \eqref{eq:wt(i,a)} and \eqref{eq:wt-one-row}, we have
  \begin{align*}
    \wt_{(n)}(\alpha)
    &=\wt_{(n-k)}(\beta)(1-q^{n-k}t^k)
      \prod_{j=n-k+1}^{n-1}(q^j-1)\\
    &=(-1)^{k-1}(1-q^{n-k}t^k)
      \frac{(q;q)_{n-1}}{(q;q)_{n-k}}\wt_{(n-k)}(\beta).
  \end{align*}
  Summing over \( \beta \) and \( k \), using
  \( e_\alpha=e_\beta e_k \), and dividing by \( (q;q)_{n-1} \),
  we obtain
  \begin{equation}\label{eq:one-row-rec}
    \frac{Y_{(n)}}{(q;q)_{n-1}}
    =\sum_{k=1}^n(-1)^{k-1}(1-q^{n-k}t^k)e_k
      \frac{Y_{(n-k)}}{(q;q)_{n-k}}.
  \end{equation}

  On the other hand, for the function \( A(v) \) in
  \Cref{lem:one-row-generating}, we have
  \[
    (1-v)A(v)=(1-tv)A(qv).
  \]
  Multiplying these identities over \( i\ge1 \) with \( v=x_i u \)
  and applying \Cref{lem:one-row-generating} gives
  \[
    \left(\prod_{i\ge1}(1-x_i u)\right)
    \sum_{n\ge0}\frac{J_{(n)}}{(q;q)_n}u^n
    =\left(\prod_{i\ge1}(1-tx_i u)\right)
    \sum_{n\ge0}\frac{J_{(n)}}{(q;q)_n}(qu)^n.
  \]
  Using \( \prod_{i\ge1}(1-x_i u)=\sum_{k\ge0}(-1)^ke_ku^k \)
  and comparing coefficients of \( u^n \), we obtain, for \( n\ge1 \),
  \[
    \frac{J_{(n)}}{(q;q)_{n-1}}
    =\sum_{k=1}^n(-1)^{k-1}(1-q^{n-k}t^k)e_k
      \frac{J_{(n-k)}}{(q;q)_{n-k}}.
  \]
  This is the same recurrence as \eqref{eq:one-row-rec}. Since
  \( J_{(0)}=Y_{(0)}=1 \), we obtain \( J_{(n)}=Y_{(n)} \) for all
  \( n\ge0 \) by induction.
\end{proof}

\section{Hook case}\label{sec:hook}

In this section, we give an \( e \)-expansion of \( J_{\mu}(x;q,t) \)
when \( \mu \) is a hook using the one-row case
in \Cref{sec:one-row-direct}. As a consequence, we prove
\Cref{conj:epositivity} in this case.

\subsection{Main formula}

We first extend the definition
of the weight \( \wt_{(n)}(\alpha) \) given in the previous section. Recall
that \( \chi(P)=1 \) if \( P \) holds and \( \chi(P)=0 \) otherwise.

\begin{defn}
  For a hook \( \mu=(r,1^s) \) with \( r\ge1 \), \( s\ge0 \), and
  \( r+s=n \), and a composition \( \alpha=(\alpha_1,\dots,\alpha_k) \)
  of \( n \), define
\begin{equation}
\label{eq:wt-hook}
    \wt_{(r,1^s)}(\alpha) = \chi(\alpha_k>s) (t;t)_s \prod_{i=1}^{r} \wt(i,\alpha).
\end{equation}
\end{defn}

The \( e \)-coefficients of \( J_{(r,1^s)}(x;q,t) \) have a combinatorial description
in terms of these weights, as follows. 

\begin{thm}\label{thm:main}
  Let \( n,r\ge1 \) and \( s\ge0 \) be integers with \( r+s=n \). Then we have
  \begin{equation} \label{eq:main-hook}
    J_{(r,1^s)}(x;q,t) =\sum_{\lambda\vdash n} \left(  \sum_{\alpha\sim \lambda} \wt_{(r,1^s)}(\alpha) \right)e_\lambda,
  \end{equation}
  where the second sum is over all compositions \( \alpha \)
  with \( \alpha\sim\lambda \).
\end{thm}

This theorem will be proved in the next subsection. In what follows,
we provide an equivalent description of \eqref{eq:main-hook} using ribbon tilings.

A \emph{ribbon} is a connected skew shape containing no
\( 2\times2 \) block. The \emph{content} of a cell \( (i,j) \)
is \( j-i \), and the \emph{head} of a ribbon is its cell of
smallest content.

A \emph{ribbon tiling} of a partition \( \nu \) is a tiling of
its diagram by ribbons. Its \emph{type} is the partition formed
by the ribbon sizes. Write \( \RT(\nu) \) for the set of ribbon
tilings of \( \nu \), and \( \RT(\nu,\lambda) \) for those of
type \( \lambda \). A ribbon tiling is \emph{special} if each
ribbon contains a cell in the first column of \( \nu \).
Denote the corresponding sets by \( \SRT(\nu) \) and
\( \SRT(\nu,\lambda) \).

For \( T\in\RT(\nu) \), define
\[
  \wt(T)=\prod_{c\in\nu}\wt(c,T,\nu),
  \qquad
  \wt(c,T,\nu)=
  \begin{cases}
    1-q^{\leg(c)}t^{|\theta|},
    &\text{if \( c \) is the head of a ribbon \( \theta\in T \)},\\
    q^{\leg(c)}-t^{\arm'(c)},&\text{otherwise}.
  \end{cases}
\]
Here the leg and coarm lengths are measured in \( \nu \).

\begin{thm}\label{thm:2}
  Let \( \mu=(r,1^s) \), where \( r\ge1 \), \( s\ge0 \), and
  \( r+s=n \). Then
  \begin{equation}\label{eq:main-ribbon}
    J_\mu(x;q,t)
    =\sum_{\lambda\vdash n}
      \left(\sum_{T\in\SRT(\mu',\lambda)}\wt(T)\right)e_\lambda.
  \end{equation}
\end{thm}

This follows directly from \Cref{thm:main} by interpreting each
composition with nonzero weight as a special ribbon tiling of \( \mu' \)
with the same weight.

\subsection{Proof of \Cref{thm:main}}

Let
\begin{equation}\label{eq:Y_mu}
  Y_{(r,1^s)} = Y_{(r,1^s)}(x;q,t)
  =\sum_{\alpha\vDash n}\wt_{(r,1^s)}(\alpha)e_\alpha.
\end{equation}
Then \Cref{thm:main} is equivalent to 
\begin{equation}\label{eq:10}
  J_{(r,1^s)}(x;q,t) = Y_{(r,1^s)}(x;q,t).
\end{equation}
It is known \cite[p.~363]{Macdonald} that for \( r,s\ge 1 \),
\begin{equation}\label{eq:rec_J}
    (1-q^r t^s)J_{(r)} J_{(1^s)}
    =(1-t^s) J_{\left(r+1,1^{s-1}\right)}+(1-q^r) J_{\left(r, 1^s\right)}.
\end{equation}
We will show that \( Y_{(r,1^s)}(x;q,t) \) satisfies the same
recurrence relation and initial conditions. Note that if
\( (r,1^s) = (1,1^{n-1})=(1^n) \), then by definition we have
\( \wt_{(1,1^{n-1})}(\alpha)=0 \) unless \( \alpha=(n) \). Thus,
\begin{equation}\label{eq:12}
  Y_{(1^n)} = (t;t)_n e_n.
\end{equation}

\begin{lem}\label{lem:5}
  For integers \( r,s\ge 1 \), we have
\begin{equation}\label{eqn:Yrec}
    (1-q^r t^s)Y_{(r)} Y_{(1^s)}
    =(1-t^s) Y_{\left(r+1,1^{s-1}\right)}+(1-q^r) Y_{\left(r, 1^s\right)}.
\end{equation}
\end{lem}

\begin{proof}
  Let \(n=r+s\). By \eqref{eq:Y_mu}, the right-hand side of
  \eqref{eqn:Yrec} is equal to
  \begin{equation}\label{eq:8}
    \sum_{\alpha\vDash n}
    \left(
      (1-t^s)\wt_{(r+1,1^{s-1})}(\alpha)
      +(1-q^r)\wt_{(r,1^s)}(\alpha)
    \right)e_\alpha.
  \end{equation}
  Fix \(\alpha=(\alpha_1,\dots,\alpha_k)\vDash n\). If
  \(\alpha_k\le s-1\), then both weights in \eqref{eq:8} vanish. If
  \(\alpha_k\ge s+1\), then \(r>\alpha_1+\cdots+\alpha_{k-1}\), so
  \(\wt(r+1,\alpha)=q^r-1\). Hence, the summand in \eqref{eq:8} also
  vanishes:
  \[
    (1-t^s)\wt_{(r+1,1^{s-1})}(\alpha)
    +(1-q^r)\wt_{(r,1^s)}(\alpha)
    =(t;t)_s\bigl((q^r-1)+(1-q^r)\bigr)
      \prod_{i=1}^r\wt(i,\alpha)=0.
  \]
  Therefore, we can restrict the sum in \eqref{eq:8} to the
  compositions with \( \alpha_k=s \).

  Suppose \(\alpha=(\beta,s)\) with \(\beta\vDash r\). Then
  \(\wt_{(r,1^s)}(\alpha)=0\) and
  \[
    (1-t^s)\wt_{(r+1,1^{s-1})}(\alpha)
    =(1-q^rt^s)(t;t)_s\wt_{(r)}(\beta).
  \]
  This, together with \(e_\alpha=e_\beta e_s\) and \eqref{eq:12},
  implies that \eqref{eq:8} is equal to
  \[
    (1-q^rt^s)(t;t)_s e_s
    \sum_{\beta\vDash r}\wt_{(r)}(\beta)e_\beta
    =(1-q^rt^s)Y_{(1^s)}Y_{(r)}.
  \]
  This is the left-hand side of \eqref{eqn:Yrec}, completing the
  proof.
\end{proof}

Now we are ready to prove \Cref{thm:main}.

\begin{proof}[Proof of \Cref{thm:main}]
  By \Cref{thm:one-row}, we have
  \[
    J_{(k)}(x;q,t) = Y_{(k)}(x;q,t) \qquad \text{for all } k\ge1.
  \]
  By \cite[p.~339]{Macdonald} and \eqref{eq:12}, we have
  \[
    J_{(1^k)}(x;q,t) = (t;t)_k e_k = Y_{(1^k)}(x;q,t) \qquad \text{for all } k\ge1.
  \]
  Thus, using \eqref{eq:rec_J}, \eqref{eqn:Yrec}, and the above
  results, we obtain \eqref{eq:10} by induction on \( s \).
\end{proof}

\subsection{Proof of \Cref{conj:epositivity} for hooks}

In this subsection we prove the \( e \)-positivity in \Cref{conj:epositivity}
when \( \mu \) is a hook.

Let \( \mu=(r,1^s) \) with \( n=r+s \), and fix an integer
\( k\ge1 \). For a composition
\( \alpha=(\alpha_1,\dots,\alpha_\ell)\vDash n \) and
\( 1\le i\le r \), define
\[
  \wt_k(i,\alpha)=
  \begin{cases}
    [k\alpha_j]_q-[i-1]_q
    &\text{if }i=1+\alpha_1+\cdots+\alpha_{j-1}
      \text{ for some }1\le j\le\ell,\\
    [i-1]_q &\text{otherwise.}
  \end{cases}
\]
Then \Cref{thm:main} with \( t=q^{-k} \) gives
\begin{equation}\label{eq:hook-large-k-positive}
    q^{k\left(n+\binom{s+1}{2}\right)}
    \frac{J_{(r,1^s)}(x;q,q^{-k})}{(q-1)^n}
    =\left(\prod_{j=1}^s[kj]_q\right)
      \sum_{\substack{\alpha\vDash n\\
    \alpha_{\ell(\alpha)}>s}}
    \left(\prod_{i=1}^r\wt_k(i,\alpha)\right)e_\alpha.
\end{equation}
If \( k\ge r-1 \), then
\[
  [k\alpha_j]_q-[i-1]_q =  q^{i-1}[k\alpha_j-i+1]_q \in \ZZ_{\ge0}[q].
\] 
Thus, \Cref{conj:epositivity} holds in this case. For an arbitrary
\( k\ge1 \), we first establish the following one-row formula,
which is different from \Cref{thm:one-row-direct} for \( t=q^{-k} \).
Throughout this section, the notation
\( \sum_{c_0+\cdots+c_{k-1}=n} \) means the sum over all nonnegative
integers \( c_0,\dots,c_{k-1} \) satisfying
\( c_0+\cdots+c_{k-1}=n \).

\begin{prop}\label{prop:one-row-positivity}
  Let \( n,k\ge1 \) be integers. Then
  \begin{equation}\label{eq:one-row-positive-specialization}
    q^{kn}\frac{J_{(n)}(x;q,q^{-k})}{(q-1)^n}
    =[n]_q! \sum_{c_0+\cdots+c_{k-1}=n}
        q^{\sum_{i=0}^{k-1}ic_i}\prod_{i=0}^{k-1}e_{c_i}.
  \end{equation}
  In particular, the left-hand side has \( e \)-coefficients in
  \( \ZZ_{\ge0}[q] \).
\end{prop}

\begin{proof}
  Fix \( k\ge1 \), and let
  \begin{equation}\label{eq:15}
    H_n=\frac{q^{kn}J_{(n)}(x;q,q^{-k})}{(q-1)^n[n]_q!},
    \qquad 
    F_n= \sum_{c_0+\cdots+c_{k-1}=n}
      q^{\sum_{i=0}^{k-1}ic_i}\prod_{i=0}^{k-1}e_{c_i}.
  \end{equation}
  Then we must show that \( H_n=F_n \).

  We first derive a recurrence for \( H_n \). By
  \eqref{eq:hook-large-k-positive} with \( r=n \) and \( s=0 \),
  we have
  \begin{equation}\label{eq:13}
    [n]_q! H_n
    = \sum_{\alpha\vDash n}
    \left(\prod_{i=1}^n\wt_k(i,\alpha)\right)e_\alpha.
\end{equation}
If \( \alpha=(\beta,a)\vDash n \), then the last block of \( \alpha \)
contributes the weight
  \[
   e_a \prod_{i=n-a+1}^{n}  \wt_k(i,\alpha) 
    = e_a ([ka]_q-[n-a]_q)\frac{[n-1]_q!}{[n-a]_q!}.
  \]
Hence, dividing \eqref{eq:13} by \( [n-1]_q! \)  gives
\begin{equation}\label{eq:hook-positive-row-rec}
  [n]_qH_n=\sum_{a=1}^n([ka]_q-[n-a]_q)e_aH_{n-a}
    \qquad(n\ge1).
  \end{equation}

  To derive the recurrence for \( F_n \), we claim that
  \begin{equation}\label{eq:14}
    \sum_{a=0}^n q^{ka}e_aF_{n-a}
    =\sum_{d_0+\cdots+d_k=n}
       q^{\sum_{i=0}^k id_i}\prod_{i=0}^k e_{d_i}
    =\sum_{a=0}^n q^{n-a}e_aF_{n-a}.
  \end{equation}
  The first equality is obtained by setting \( d_k=a \) and
  \( d_i=c_i \) for \( 0\le i\le k-1 \). The second is obtained by
  setting \( d_0=a \) and \( d_i=c_{i-1} \) for \( 1\le i\le k \),
  since
  \[
    \sum_{i=1}^k id_i
    =\sum_{i=0}^{k-1}(i+1)c_i
    =n-a+\sum_{i=0}^{k-1}ic_i.
  \]
  Solving for \( F_n \) in \eqref{eq:14} yields
  \[
    (q^n-1)F_n
    =\sum_{a=1}^n(q^{ka}-q^{n-a})e_aF_{n-a}.
  \]
  Dividing by \( q-1 \) shows that \( F_n \) satisfies the same
  recurrence \eqref{eq:hook-positive-row-rec}. Since \( F_0=H_0=1 \),
  induction gives \( F_n=H_n \) for all \( n\ge0 \). 
\end{proof}

We now use the one-row formula to obtain a positive formula for
every hook.

\begin{cor}\label{cor:hook-positivity}
  Let \( r,k\ge1 \) and \( s\ge0 \) be integers with \( n=r+s \).
  Then
    \begin{multline}\label{eq:hook-positive-formula}
    \qquad\qquad  q^{k\left(n+\binom{s+1}{2}\right)}
      \frac{J_{(r,1^s)}(x;q,q^{-k})}{(q-1)^n}\\
      =[r-1]_q!\prod_{j=1}^s[kj]_q
      \sum_{\substack{d_0+\cdots+d_k=n\\ d_k\le s<d_0}}
        q^{\sum_{i=0}^k i d_i}[k(d_0-d_k)]_q
        \prod_{i=0}^k e_{d_i}.\qquad\qquad\quad
      \end{multline}
  In particular, the expression in \eqref{eq:Schurpositivity} has
  \( e \)-coefficients in \( \ZZ_{\ge0}[q] \) for every hook
  partition.
\end{cor}

\begin{proof}
  Fix \( k\ge1 \). Let
  \[
    B_{r,s}=q^{k\left(n+\binom{s+1}{2}\right)}
       \frac{J_{(r,1^s)}(x;q,q^{-k})}{(q-1)^n},
    \qquad C_s=\prod_{j=1}^s[kj]_q.
  \]
  By the same argument as in the proof of \Cref{prop:one-row-positivity},
  separating the last part \( a>s \) of the compositions
  \( \alpha=(\beta,a)\vDash n \) in \eqref{eq:hook-large-k-positive}
  gives
  \begin{align*}
    B_{r,s}
    &=C_s \sum_{a=s+1}^n  
    e_a ([ka]_q-[n-a]_q)\frac{[r-1]_q!}{[n-a]_q!}
      \sum_{\beta\vDash n-a}
    \left(\prod_{i=1}^{n-a}\wt_k(i,\beta)\right)e_\beta\\
    &=C_s[r-1]_q! \sum_{a=s+1}^n([ka]_q-[n-a]_q)e_a \frac{B_{n-a,0}}{[n-a]_q!}.
  \end{align*}
  Applying \Cref{prop:one-row-positivity} to \( B_{n-a,0} \) and
  using \( [ka]_q-[n-a]_q=(q^{ka}-q^{n-a})/(q-1) \), we obtain
  \begin{equation}\label{eq:16}
    \frac{B_{r,s}}{C_s[r-1]_q!}
    = \sum_{a=s+1}^n \frac{q^{ka}-q^{n-a}}{q-1} e_a
    \sum_{c_0+\cdots+c_{k-1}=n-a}
      q^{\sum_{i=0}^{k-1}i c_i} \prod_{i=0}^{k-1} e_{c_i}.
  \end{equation}

  Consider a tuple \( (c_0,\dots,c_{k-1}) \) of nonnegative integers
  with \( c_0 + \cdots + c_{k-1}=n-a \).
  If we let \( (d_0,\dots,d_k) = (c_0,\dots,c_{k-1},a) \),
  then \( d_0 + \cdots + d_k=n \), \( d_k=a>s \), and
  \[
    \sum_{i=0}^k i d_i  = ka+ \sum_{i=0}^{k-1} i c_i.
  \]
  If we let \( (d_0,\dots,d_k) = (a,c_0,\dots,c_{k-1}) \),
  then \( d_0 + \cdots + d_k=n \), \( d_0=a>s \), and
  \[
    \sum_{i=0}^k i d_i  = \sum_{i=0}^{k-1} (i+1) c_i
    = (n-a)+ \sum_{i=0}^{k-1} i c_i.
  \]
  Thus, we can rewrite \eqref{eq:16} as
  \begin{equation}\label{eq:17}
    \frac{B_{r,s}}{C_s[r-1]_q!}
    = \sum_{\substack{d_0+\cdots+d_{k}=n\\ d_k>s}}
      \frac{q^{\sum_{i=0}^{k}i d_i}}{q-1} \prod_{i=0}^{k} e_{d_i}
      -    \sum_{\substack{d_0+\cdots+d_{k}=n\\ d_0>s}}
      \frac{q^{\sum_{i=0}^{k}i d_i}}{q-1} \prod_{i=0}^{k} e_{d_i}.
  \end{equation}
  Interchanging \( d_0 \) and \( d_k \) gives
  \[
    \sum_{\substack{d_0+\cdots+d_{k}=n\\ d_k>s}}
      q^{\sum_{i=0}^{k}i d_i} \prod_{i=0}^{k} e_{d_i}
      =\sum_{\substack{d_0+\cdots+d_{k}=n\\ d_0>s}}
      q^{kd_0+\sum_{i=0}^{k-1}i d_i} \prod_{i=0}^{k} e_{d_i}.
  \]
  Thus \eqref{eq:17} is equivalent to
  \[
    \frac{B_{r,s}}{C_s[r-1]_q!}
    = \sum_{\substack{d_0+\cdots+d_{k}=n\\ d_0>s}}
     \frac{q^{kd_0}-q^{kd_k}}{q-1}
      q^{\sum_{i=0}^{k-1}i d_i} \prod_{i=0}^{k} e_{d_i}.
  \]
  Since the summand is antisymmetric in \( d_0 \) and \( d_k \), the
  sum over nonnegative integers \( d_i \) satisfying both \( d_0>s \)
  and \( d_k>s \) vanishes. Thus, we obtain
  \begin{align*}
    \frac{B_{r,s}}{C_s[r-1]_q!}
    &= \sum_{\substack{d_0+\cdots+d_{k}=n\\ d_k\le s< d_0}}
     \frac{q^{kd_0}-q^{kd_k}}{q-1}
      q^{\sum_{i=0}^{k-1}i d_i} \prod_{i=0}^{k} e_{d_i}\\
    &= \sum_{\substack{d_0+\cdots+d_{k}=n\\ d_k\le s< d_0}}
     q^{kd_k}[k(d_0-d_k)]_q 
      q^{\sum_{i=0}^{k-1}i d_i} \prod_{i=0}^{k} e_{d_i},
  \end{align*}
  which is the desired formula.
\end{proof}

\section{Two-column case}\label{sec:two-column}

In this section, we provide a formula for the \( e \)-coefficients of
\( J_\mu (x;q,t) \) and prove \Cref{conj:epositivity} when \( \mu \)
consists of two columns. In this case, each nonzero
\( e \)-coefficient of \( J_\mu (x;q,t) \) admits a factorization
into \( |\mu| \) factors.

\begin{thm}\label{thm:two-column}
  Let \( 0\le\ell\le\lfloor n/2\rfloor \) and
  \( \mu=(2^\ell, 1^{n-2\ell}) \). Then we have
  \begin{equation}\label{eq:18}
    J_\mu(x;q,t)
    = \sum_{p=0}^{\ell}
    \left( q^{\ell -p}(1-t^{n-2p}) (t;t)_{n-\ell-p-1} (q^{-1};t)_{\ell-p}
  \frac{(t;t)_{\ell}}{(t;t)_{\ell -p}}\frac{(qt;t)_{n-\ell}}{(qt;t)_{n-\ell -p}}
 \right) e_{(n-p,p)},
  \end{equation}
  where we use the convention that \( (1-t^{n-2p})(t;t)_{n-\ell-p-1}=(t;t)_{n-2p} \) when \( p=\ell \).

  Equivalently, we have
  \[
    J_\mu (x;q,t) =
    \sum_{p=0}^{\ell} \left( \prod_{c\in\mu}\wt_{\mu,p}(c) \right) e_{(n-p,p)},
  \]
  where the weight \( \wt_{\mu,p}(c) \)
  of a cell \( c=(i,j)\in\mu \) is defined by
  \[
    \wt_{\mu,p}(c)=
    \begin{cases}
      1-q^{\arm(c)}t^{\leg(c)+1}, & \mbox{if \( i\le p \)},\\
      1-t^{n-2p}, & \mbox{if \( c=(p+1,1) \)},\\
      q^{j-1}-t^{i-p-1}, & \mbox{if \( i>p \) and \( c\ne(p+1,1) \)},
    \end{cases}
  \]
  with arm and leg lengths measured in \( \mu \).
  See \Cref{fig:two-column-cell-factors}.
\end{thm}

\begin{figure}
  \centering
  \begin{tikzpicture}[x=1.45cm,y=.65cm,
      cell/.style={draw=black,line width=.5pt},
      factor/.style={font=\small},
      note/.style={align=left,font=\small}]
    \foreach \r/\leftfactor/\rightfactor in {
      1/{1-qt^6}/{1-t^4},
      2/{1-qt^5}/{1-t^3}} {
      \draw[cell] (0,\r-1) rectangle (1,\r);
      \draw[cell] (1,\r-1) rectangle (2,\r);
      \node[factor] at (.5,\r-.5) {$\leftfactor$};
      \node[factor] at (1.5,\r-.5) {$\rightfactor$};
    }
    \draw[cell,fill=red!12] (0,2) rectangle (1,3);
    \node[factor] at (.5,2.5) {$1-t^6$};
    \foreach \r/\factor in {4/{1-t},5/{1-t^2},6/{1-t^3}} {
      \draw[cell] (0,\r-1) rectangle (1,\r);
      \node[factor] at (.5,\r-.5) {$\factor$};
    }
    \draw[cell] (1,2) rectangle (2,3);
    \node[factor] at (1.5,2.5) {$q-1$};
    \draw[cell] (1,3) rectangle (2,4);
    \node[factor] at (1.5,3.5) {$q-t$};
    \draw[decorate,decoration={brace,amplitude=4pt}]
      (2.3,6) -- (2.3,2.12);
    \node[note,anchor=west] at (2.55,4.1)
      {Remaining diagram $\nu=(2^2,1^2)$\\[3pt]
       $q^{\arm'(c)}-t^{\leg'(c)}$, except at the southwest cell\\[3pt]
       Southwest cell: $1-t^{|\nu|}=1-t^6$};
    \draw[decorate,decoration={brace,amplitude=4pt}]
      (2.3,1.88) -- (2.3,0);
    \node[note,anchor=west] at (2.55,.95)
      {Bottom $p=2$ rows\\[3pt]
       $1-q^{\arm(c)}t^{\leg(c)+1}$, measured in $\mu$};
  \end{tikzpicture}
  \caption{The weight \( \wt_{\mu,p}(c) \) of each cell \( c\in \mu \)
    for \( \mu=(2^4,1^2) \) and \( p=2 \).}
  \label{fig:two-column-cell-factors}
\end{figure}

To prove \Cref{thm:two-column}, we use the following lemma.

\begin{lem}\label{lem:Pieri}
  For \( a\ge b \) with \( n=a+b \), we have
  \begin{equation}\label{eqn:Pieri}
    e_a \cdot e_b=\sum_{j=0}^b \frac{(q;t)_{a-j}(q;t)_{b-j}(1-qt^{n-2j})}{(t;t)_{a-j}(t;t)_{b-j}(t;t)_j(q;t)_{n-j+1} }J_{(2^j,1^{n-2j})}(x;q,t).
  \end{equation}
\end{lem}

\begin{proof}
  Note that \( P_{(1^n)}(x;q,t)=e_n(x) \). Then the lemma follows
  directly from the Pieri rule for Macdonald polynomials in
  \cite[Ch.~VI,~(6.24)~(iv)]{Macdonald} by taking \( \mu \) to be a
  one-column partition.
\end{proof}

\begin{proof}[Proof of \Cref{thm:two-column}]
  We use \Cref{lem:Pieri} to verify the \( e \)-coefficients of
  \( J_{(2^\ell,1^{n-2\ell})} (x;q,t) \). Note that the identities
  \eqref{eqn:Pieri} with \( (a,b)=(n-\ell,\ell) \) for
  \( 0\le \ell\le \lfloor n/2\rfloor \) determine the polynomials
  \( J_{(2^\ell,1^{n-2\ell})}(x;q,t) \) uniquely. Indeed, the
  coefficient of \( J_{(2^\ell,1^{n-2\ell})} \) in \eqref{eqn:Pieri}
  when \( j=\ell \) is
  \[
    \frac{(q;t)_{n-2\ell}\,(1-qt^{n-2\ell})}
         {(t;t)_{n-2\ell}\,(t;t)_{\ell}\,(q;t)_{n-\ell+1}} \ne 0,
  \]
  so \eqref{eqn:Pieri} with \( (a,b)=(n-\ell,\ell) \) expresses
  \( J_{(2^\ell,1^{n-2\ell})} \) in terms of \( e_{n-\ell}\,e_\ell \)
  and the \( J_{(2^j,1^{n-2j})} \) with \( j<\ell \), and induction on
  \( \ell \) proves the claim. Therefore, it suffices to show that the
  right-hand sides of \eqref{eq:18} satisfy \eqref{eqn:Pieri} with
  \( (a,b)=(n-\ell,\ell) \) for every
  \( 0\le \ell\le \lfloor n/2\rfloor \), that is,
  \begin{equation}\label{eq:19}
   e_{(n-\ell,\ell)}(x)=\sum_{p=0}^\ell S_{n,\ell,p} e_{(n-p,p)}(x),
 \end{equation}
 where
 \[
   S_{n,\ell,p} =   \sum_{j=p}^\ell q^{j-p}\frac{(1-t^{n-2p})(1-qt^{n-2j})}{(1-q)}\frac{(q^{-1};t)_{j-p}(t;t)_{n-j-p-1}(q;t)_{n-\ell -j}(q;t)_{\ell -j}}{(t;t)_{j-p}(qt;t)_{n-j-p}(t;t)_{n-\ell -j}(t;t)_{\ell -j}} .
 \]
 Since \(\{ e_\lambda\} \) forms a basis, \eqref{eq:19} is equivalent to
 \( S_{n,\ell,p} = \chi(p=\ell) \). The case \( p=\ell \) is
 immediate, so it remains to show that \( S_{n,\ell,p} = 0 \) for
 \( 0\le p<\ell \le \lfloor n/2\rfloor \).

  Reversing the summation by replacing \( j \) with \( \ell-j \)
  and using the identities
  \[
    (a;t)_{r+j}=(a;t)_r(at^r;t)_j,
    \qquad
    (a;t)_{r-j}
      =\frac{(a;t)_r}
             {(-a)^j t^{(r-1)j-\binom{j}{2}}(a^{-1}t^{1-r};t)_j},
  \]
  where the second identity holds for \( 0\le j\le r \), we obtain
  \begin{align*}
    S_{n,\ell,p}
      &=\frac{q^{\ell-p}(1-t^{n-2p})(q^{-1};t)_{\ell-p}
               (t;t)_{n-\ell-p-1}(q;t)_{n-2\ell}}
              {(1-q)(t;t)_{\ell-p}(qt;t)_{n-\ell-p}(t;t)_{n-2\ell}}\\
      &\quad\times\sum_{j=0}^{\ell-p}t^j
        \frac{(q,t^{p-\ell},t^{n-\ell-p},qt^{n-2\ell};t)_j
               (1-qt^{n-2\ell+2j})}
             {(t,qt^{p-\ell+1},qt^{n-\ell-p+1},t^{n-2\ell+1};t)_j}.
  \end{align*}
  We apply the terminating very-well-poised \( {}_6\phi_5 \)
  summation formula \cite[(II.21)]{GR2004} with base \( t \) and
  parameters
  \[
    a=qt^{n-2\ell},\qquad b=t^{n-\ell-p},\qquad c=q,
  \]
  and terminating parameter \( t^{p-\ell} \).
  Since \( at^{\ell-p+1}/(bc)=t \), the sum above equals
  \[
    (1-a)\frac{(at,at/(bc);t)_{\ell-p}}
                   {(at/b,at/c;t)_{\ell-p}}
      =(1-qt^{n-2\ell})
        \frac{(qt^{n-2\ell+1},t^{p-\ell+1};t)_{\ell-p}}
             {(qt^{p-\ell+1},t^{n-2\ell+1};t)_{\ell-p}}
      =0.
  \]
  Indeed, since \( \ell-p\ge1 \), the factor
  \( (t^{p-\ell+1};t)_{\ell-p} \) contains \( 1-t^0=0 \).
  Thus \( S_{n,\ell,p}=0 \), completing the proof.
\end{proof}

Using \Cref{thm:two-column}, we can prove \Cref{conj:epositivity} when
\( \mu \) has at most two columns.

\begin{cor}\label{cor:two-column-positivity}
  For every partition \( \mu \) with at most two columns and every
  positive integer \( k \), the expression in
  \eqref{eq:Schurpositivity} has \( e \)-coefficients in
  \( \ZZ_{\ge0}[q] \).
\end{cor}

\begin{proof}
  For \( \mu=(2^\ell,1^{n-2\ell}) \), consider the coefficients
  \[
    A_j=[e_{(n-j,j)}]\left(
      q^{k\sum_i i\mu_i}\frac{J_\mu(x;q,q^{-k})}{(q-1)^n}\right),
    \qquad 0\le j\le\ell.
  \]
  Substituting \( t=q^{-k} \) into \Cref{thm:two-column} and using
  \( 1-q^{-a}=q^{-a}(q-1)[a]_q \), we obtain, for
  \( 0\le j<\ell \),
  \[
      A_j=q^j[k(n-2j)]_q
        \prod_{i=0}^{\ell-j-1}[ki+1]_q
        \prod_{i=1}^{n-\ell-j-1}[ki]_q
       \prod_{i=\ell-j+1}^{\ell}[ki]_q
        \prod_{i=n-\ell-j+1}^{n-\ell}[ki-1]_q.
  \]
  For \( j=\ell \), the convention in \Cref{thm:two-column} gives
  \[
    A_\ell=q^\ell
      \prod_{i=1}^{n-2\ell}[ki]_q
      \prod_{i=1}^{\ell}[ki]_q
      \prod_{i=n-2\ell+1}^{n-\ell}[ki-1]_q.
  \]
  Thus \( A_j\in\ZZ_{\ge0}[q] \) for every \( j \), proving the
  corollary.
\end{proof}

\section{The \( e_n \)-coefficient}\label{sec:one-row}

Combining the relation between \( P_\mu(x;q,t) \) and \( J_\mu(x;q,t) \) with the triangularity
of \( P_\mu(x;q,t) \) in the monomial symmetric function basis, we obtain  
  \[
   J_\mu(x;q,t)  = \sum_{\lambda\ge \mu'} c_{\mu,\lambda} e_\lambda,
  \]
  where
  \begin{equation}\label{eq:cmuprime}
   c_{\mu, \mu'} =  \prod_{c\in \mu} (1-q^{\arm(c)} t^{\leg(c)+1})
   =  \prod_{c\in \mu'} (1-q^{\leg(c)} t^{\arm(c)+1}).
  \end{equation}
  In particular,
  \[
    J_{(1^n)}(x;q,t) = (t;t)_n e_n.
  \]
  More generally, we compute the \( e_n \)-coefficient of
  \( J_\mu(x;q,t) \) for any \( \mu\vdash n \).

\begin{thm}\label{thm:e_n}
  For \( n\ge 1 \) and \( \mu\vdash n \), we have
  \begin{align*}
    [e_n] J_\mu(x;q,t)
    &=(1-t^n)(q-1) \cdots (q^{\mu_1-1}-1)
    \prod_{c \notin \Row_1(\mu)}(q^{\arm'(c)}- t^{\leg'(c)})\\
&=(1-t^n)(q-1) \cdots (q^{\mu_1-1}-1)
    \prod_{c \notin \Col_1(\mu')}(q^{\leg'(c)}-  t^{\arm'(c)}),
\end{align*}
where \( \Row_1(\mu) \) (resp.~\( \Col_1(\mu) \)) is the set of cells
in the first row (resp.~column) of \( \mu \).
Equivalently,
\begin{equation}\label{eq:9}
    [e_n] J_\mu(x;q,t)
    =(1-t^n) \prod_{c\in \mu'\setminus \{(1,1)\}}(q^{\leg'(c)}-  t^{\arm'(c)}).
\end{equation}
\end{thm}

To prove \Cref{thm:e_n}, we need the following lemma.

\begin{lem}\label{lem:a_n}
Let \( n\ge 1 \). If
\[
  f=\sum_{\lambda\vdash n} a_\lambda e_\lambda,
\]
then
\begin{equation}\label{eq:en-extraction}
  [e_n]f
  =
  (-1)^{n-1}(1-t^n)
  \lim_{u\to 1}f\left[ \frac{1-u}{1-t} \right]/ (1-u).
\end{equation}
  \end{lem}
  \begin{proof}
By \cite[Ch.~I, (2.14$'$)]{Macdonald}, we have
\begin{equation}\label{eq:20}
  e_m \left[ \frac{1-u}{1-t} \right]=\sum_{\lambda\vdash m}
  (-1)^{m-\ell(\lambda)}z_\lambda^{-1}p_\lambda \left[ \frac{1-u}{1-t} \right],
\end{equation}
where \( z_\lambda \) is given in \eqref{eq:z_la}.
Since
\[
  p_k\left[ \frac{1-u}{1-t} \right]=\frac{1-u^k}{1-t^k}
  =\frac{k(1-u)}{1-t^k}+O((1-u)^2),
\]
we see that only the term \(\lambda=(m)\) of the right-hand side of
\eqref{eq:20} contributes to first order in \(1-u\). Therefore,
\begin{align*}
  f \left[ \frac{1-u}{1-t} \right] &= a_{(n)}e_n \left[ \frac{1-u}{1-t} \right] +
  \sum_{\lambda< (n)}a_\lambda e_\lambda \left[ \frac{1-u}{1-t} \right]\\
  &= a_{(n)} (-1)^{n-1}\frac{1-u}{1-t^n} + O((1-u)^2). 
\end{align*}
Then the result follows by taking the limit as \( u \) approaches \( 1 \) and solving
for the coefficient \( a_{(n)} \).
  \end{proof}

  \begin{proof}[Proof of \Cref{thm:e_n}]
    By \cite[Ch.~VI, (8.8)]{Macdonald}, we have
\[
  J_\mu\left[\frac{1-u}{1-t};q,t\right]
  =\prod_{c\in\mu} \left(t^{\leg'(c)}-q^{\arm'(c)}u\right).
\]
Since the cell \( (1,1) \) corresponds to the factor \( (1-u) \), we have 
\[
  \lim_{u\to 1}J_\mu\left[ \frac{1-u}{1-t} \right]/ (1-u)
  =
  \prod_{c\in\mu\setminus\{(1,1)\}}
  \left(t^{\leg'(c)}-q^{\arm'(c)}\right).
\]
By applying \Cref{lem:a_n}, we obtain 
\begin{align*}
  [e_n]J_\mu(x;q,t)
  &=
  (-1)^{n-1}(1-t^n)
  \prod_{c\in\mu\setminus\{(1,1)\}}
  \left(t^{\leg'(c)}-q^{\arm'(c)}\right)  \\
  &=
  (1-t^n)
  \prod_{c\in\mu\setminus\{(1,1)\}}
  \left(q^{\arm'(c)}-t^{\leg'(c)}\right).
\end{align*}
Since all cells in the first row have coleg length \( 0 \), we obtain
the first identity. The remaining expressions follow directly from
analyzing the diagram of \( \mu \).
\end{proof}
  
\section{Final remarks}
\label{sec:chrom-symm-funct}

In this section, we explain the analogy between the normalization of
\( J_\mu(x;q,t) \) in \Cref{conj:epositivity} and the
\( \alpha \)-chromatic symmetric functions introduced by Haglund, Oh,
and Yoo \cite{HOY2025}. We then propose a conjecture on an expansion
of the normalized \( J_\mu(x;q,t) \) in terms of chromatic symmetric
functions when \( q=1 \).

For a graph \( G \) with vertex set \( V=\{v_1,\dots,v_n\} \), the
\emph{chromatic symmetric function} \( \chi_G(x) \) is defined by
\[
  \chi_G(x) = \sum_{c:V\to \ZZ^+} x_{c(v_1)} \cdots x_{c(v_n)},
\]
where the sum is over all proper colorings \( c \) of \( G \). If
\( G \) is the incomparability graph of a natural unit interval order,
then it can be identified with a Dyck path \( \gamma \) of length
\( 2n \). From now on, we consider only such graphs and write
\( \chi_\gamma(x) \) in place of \( \chi_G(x) \). See \cite{Stan95, HOY2025}
for more details.

By \eqref{eq:J=H}, the normalization of \( J_\mu(x;q,t) \) in
\Cref{conj:epositivity} satisfies
\[
  q^{k\sum_{i\ge1} i\mu_i} \frac{J_\mu(x;q,q^{-k})}{(q-1)^n} =
  \frac{\tilde H_\mu[(q^k-1)X;q,q^k]}{(q-1)^n}.
\]
Applying the same plethystic substitution to the unicellular LLT polynomial
\( \LLT_\gamma(x;q) \) and dividing by \( (q-1)^n \)
gives the \( \alpha \)-chromatic symmetric function at \( \alpha=k \):
\[
  \chi_\gamma^{(k)}(x;q) =\frac{\LLT_\gamma[(q^k-1)X;q]}{(q-1)^n}.
\]

When \( q=1 \), the \( \alpha \)-chromatic symmetric functions are
polynomials in \( \alpha \) whose coefficients are symmetric
functions. Hence, they can be expanded in terms of the basis
\( \{\binom{\alpha}{r}\}_{r\ge0} \) of the polynomials in
\( \alpha \). Haglund, Oh, and Yoo \cite[Theorem~4.4]{HOY2025} showed
that each coefficient of \( \binom{\alpha}{r} \) has a positive expansion 
 in terms of chromatic symmetric functions:
\begin{equation}\label{eq:21}
  \chi_\gamma^{(\alpha)}(x;1)
  =\sum_{r=1}^n\binom{\alpha}{r} r!
  \sum_{S\in\mathcal{S}(n,r)}\chi_{\beta(\gamma,S)}(x).
\end{equation}
Here \( \mathcal{S}(n,r) \) denotes the set of all set partitions of
\( \{1,2,\dots,n\} \) into \( r \) blocks, and \( \beta(\gamma,S) \)
is a certain Dyck path determined by \( \gamma \) and \( S \). By
Hikita's theorem \cite{Hikita2025}, this implies that
\( \chi_\gamma^{(\alpha)}(x;1) \) is \( e \)-positive for every
positive integer \( \alpha \).

Now we define
\[
  \hat{J}^{(\alpha)} _\mu(x)= \lim_{q\rightarrow 1}\frac{J_\mu (x;q,q^{-\alpha})}{(q-1)^{|\mu|}}.
\]
Using the Haglund--Haiman--Loehr formula
\cite[Proposition~8.1]{HHL2005}, one can see that this is also a
polynomial in \( \alpha \). Inspired by \eqref{eq:21}, we propose the
following question.

\begin{question}\label{conj:jack-chromatic}
  Let \( n,r \) be positive integers, \( \mu\vdash n \) and \( 0\le r\le n \).
  Denote by \( \mathcal{OP}(n,r) \) the set of ordered set partitions of
  \( \{1,2,\dots,n\} \) into \( r \) blocks. For each
  \( S\in \mathcal{OP}(n,r) \), can we find a Dyck path
  \( \pi(\mu, S) \) such that
  \begin{equation}\label{eq:22}
    \hat{J}^{(\alpha)} _\mu(x)= \sum_{r=1}^n \binom{\alpha}{r} \sum_{S\in \mathcal{OP}(n,r)}
    \chi_{\pi(\mu, S)}(x)\ ?
  \end{equation}
\end{question}

Again by the result of Hikita \cite{Hikita2025}, an affirmative answer
to this question would imply the \( e \)-positivity of \( \hat{J}^{(\alpha)}_\mu (x) \) for
every positive integer \( \alpha \).

We note that the chromatic symmetric functions \( \{ \chi_\gamma(x)\}_{\gamma\in \mathcal{D}} \),
where \( \mathcal{D} \) is the set of Dyck paths, do not form a basis of the space of
symmetric functions. Thus, an expansion of the form \eqref{eq:22} need not be
unique. Note also that if \( \pi(\mu,S) \) depends only on \( \mu \)
and the underlying set partition of \( S \), then one can rewrite
\eqref{eq:22} as
\[
  \hat{J}^{(\alpha)} _\mu(x)= \sum_{r=1}^n \binom{\alpha}{r} r!
  \sum_{S\in \mathcal{S}(n,r)} \chi_{\pi(\mu, S)}(x),
\]
which has the same form as \eqref{eq:21}. However, such a rewriting
is not possible in general. For example, if \( \mu=(2,2,1) \), then
the coefficient of \( \binom{\alpha}{4} \) has no expansion in terms of
chromatic symmetric functions in which every coefficient is a
nonnegative integer multiple of \( 4! \).

\bibliographystyle{abbrv}

\end{document}